\nonstopmode
\documentclass[10pt]{amsart}
\usepackage{latexsym}
\usepackage{fancyhdr}
\usepackage{amsmath, amssymb}
\usepackage[ansinew]{inputenc}
\usepackage[all]{xy}
\usepackage{pdflscape}
\usepackage{longtable}
\usepackage{rotating}
\usepackage{verbatim}
\usepackage{hyperref}
\usepackage{subfigure}
\usepackage{mathrsfs}
\usepackage{mdwlist}

\theoremstyle{plain}
\newtheorem*{lemma*}{Lemma}
\newtheorem{lemma}{Lemma}
\newtheorem*{theorem*}{Theorem}
\newtheorem{theorem}{Theorem}
\newtheorem*{proposition*}{Proposition}
\newtheorem{proposition}{Proposition}
\newtheorem*{corollary*}{Corollary}
\newtheorem{corollary}{Corollary}
\newtheorem*{claim*}{Claim}

\newtheorem*{conjecture*}{Conjecture}

\newtheorem*{question*}{Question}
\theoremstyle{definition}
\newtheorem*{definition*}{Definition}

\newtheorem*{example*}{Example}

\newtheorem*{algorithm*}{Algorithm}
\newtheorem*{remark*}{Remark}
\newtheorem*{remarks*}{Remarks}
\newtheorem{remark}{Remark}

\newtheorem*{convention*}{Convention}

\numberwithin{equation}{section}

\def\al{\alpha}

\def\rh{\rho}

\def\si{\sigma}

\def\ta{\tau}

\def\vh{\varphi}

\def\om{\omega}

\def\La{\Lambda}

\def\C{\mathbb{C}}

\def\N{\mathbb{N}}

\def\R{\mathbb{R}}

\def\cE{\mathcal{E}}

\def\cH{\mathcal{H}}

\def\cO{\mathcal{O}}

\def\fM{\mathfrak{M}}
\def\fN{\mathfrak{N}}

\def\fW{\mathfrak{W}}

\def\p{\partial}

\def\<{\langle}
\def\>{\rangle}

\let\on=\operatorname

\newcommand{\sr}[1]%
{\ifmmode{}^\dagger\else${}^\dagger$\fi\ifvmode
\vbox to 0pt{\vss
 \hbox to 0pt{\hskip\hsize\hskip1em
 \vbox{\hsize3cm\raggedright\pretolerance10000
 \noindent #1\hfill}\hss}\vss}\else
 \vadjust{\vbox to0pt{\vss%
 \hbox to 0pt{\hskip\hsize\hskip1em%
 \vbox{\hsize3cm\raggedright\pretolerance10000%
 \noindent #1\hfill}\hss}\vss}}\fi%
}

\def\A{\;\forall}
\def\E{\;\exists}

\providecommand{\mapsfrom}{\kern.2em%
\setbox0=\hbox{$\leftarrow$\kern-.10em\rule[0.26mm]{0.1mm}{1.3mm}}\box0%
\kern.3em}

\title[On the Borel mapping in the quasianalytic setting]
{On the Borel mapping\\ in the quasianalytic setting}

\author[A.~Rainer]{Armin Rainer}

\address{A.~Rainer: 
Fakult\"at f\"ur Mathematik, Universit\"at Wien, 
Oskar-Morgenstern-Platz~1, A-1090 Wien, Austria}
\email{armin.rainer@univie.ac.at}

\author[G.~Schindl]{Gerhard Schindl}

\address{G.~Schindl: Fakult\"at f\"ur Mathematik, Universit\"at Wien, 
Oskar-Morgenstern-Platz~1, A-1090 Wien, Austria}
\email{gerhard.schindl@univie.ac.at}

\begin{document}

\begin{abstract}
  The Borel mapping takes germs at $0$ of smooth functions to the sequence of iterated partial derivatives at $0$.
  We prove that the Borel mapping restricted to the germs of any quasianalytic ultradifferentiable 
  class strictly larger than the real analytic class is never onto the corresponding 
  sequence space. 
\end{abstract}

\thanks{The authors were supported by FWF-Project P~26735-N25}
\date{October 12, 2015}

\maketitle

\section{Introduction}

It is a classical result due to Carleman \cite{Carleman23}, \cite{Carleman26} that the Borel mapping that takes germs at $0$ of 
functions in a quasianalytic Denjoy--Carleman class $\cE^{\{M\}}$ to the sequence of iterated partial derivatives at $0$ 
is never \emph{onto} the corresponding sequence space $\La^{\{M\}}$ unless $\cE^{\{M\}}$ is contained in the 
real analytic class. 
Here $M=(M_k)$ is a weight sequence that dominates the growth of the iterated partial derivatives of the functions 
in $\cE^{\{M\}}$, and \emph{quasianalytic} means that the Borel mapping is injective on $\cE^{\{M\}}$ 
(precise definitions will be given below). Carleman's proof is based on his formula for reconstructing the function $f \in \cE^{\{M\}}$ from the 
sequence of its iterated derivatives at $0$ (due to quasianalyticity $f$ is unique); see also \cite{Thilliez08} for 
a modern account of the proof.   

In the recent paper \cite{BonetMeise13} Bonet and Meise prove this result (non-surjectivity of the Borel mapping) for proper 
quasianalytic classes $\cE^{\{\om\}}$ (and $\cE^{(\om)}$); the brackets $\{~\}$ and $(~)$ refer to classes of Roumieu and 
Beurling type, respectively (see definitions below).  
These classes were introduced 
by Beurling \cite{Beurling61} and Bj\"orck \cite{Bjoerck66}, 
by imposing decay conditions at infinity for the Fourier transform in terms of a weight function 
$\om$, and they were equivalently described by Braun, Meise, and Taylor \cite{BMT90}. We shall refer to these classes as 
Braun--Meise--Taylor classes. In \cite{BonetMeise13} the problem is transferred to weighted spaces of entire functions via the
Fourier--Laplace transform and functional-analytic methods are applied.

Sometimes $\cE^{\{\om\}} = \cE^{\{M\}}$ (and $\cE^{(\om)} = \cE^{(M)}$) for a suitable 
sequence $M$, but in general the families of classes described by weight functions $\om$ and those described 
by weight sequences $M$ are mutually distinct; see \cite{BMM07}. 
However, the method developed in \cite{RainerSchindl12} allows us to describe the classes $\cE^{\{\om\}}$ (and $\cE^{(\om)}$) 
as unions (or intersections) of associated one parameter families of Denjoy--Carleman classes $\cE^{\{W^x\}}$ (or $\cE^{(W^x)}$), 
where $W^x$ are weight sequences associated with $\om$ in a precise way depending on a real parameter $x$. 
More generally, this construction can be turned into a definition, and in this manner one obtains ultradifferentiable 
classes $\cE^{\{\fM\}}$ (and $\cE^{(\fM)}$) defined in terms of \emph{weight matrices} $\fM = \{M^x\}_x$.  
These comprise the classical Denjoy--Carleman classes $\cE^{\{M\}}$, $\cE^{(M)}$, the Braun--Meise--Taylor classes 
$\cE^{\{\om\}}$, $\cE^{(\om)}$, and many more; cf.\ \cite[Theorem 5.22]{RainerSchindl12}. This new technique often makes it possible to treat all these classes uniformly, 
while previously every setting required a special proof.         

In the present paper we will show in an elementary way that the Borel mappings 
\begin{align*}
  j^\infty : \cE^{\{\fM\}}_{0,n} \to \La^{\{\fM\}}_n \quad \text{and} \quad 
  j^\infty : \cE^{(\fM)}_{0,n} \to \La^{(\fM)}_n, \quad j^\infty f = (\p^\al f(0))_\al,
\end{align*}
are never surjective in the proper quasianalytic setting (\emph{proper} means not contained in the real analytic class).
Here $\cE^{\{\fM\}}_{0,n}$ (resp.\ $\cE^{(\fM)}_{0,n}$) denotes the ring of germs at $0 \in \R^n$ of functions 
in $\cE^{\{\fM\}}$ (resp.\ $\cE^{(\fM)}$), and $\La^{\{\fM\}}_n$ (resp.\ $\La^{(\fM)}_n$) is the corresponding sequence space. 
As a corollary we recover the results of Bonet and Meise \cite{BonetMeise13}. 

We actually show more:
if $\cE^{\{\fM\}}$ (resp.\ $\cE^{(\fM)}$) is a proper quasianalytic class 
then 
there exist elements in $\La^{\{\fM\}}_n$ (resp.\ $\La^{(\fM)}_n$) that are not contained in 
\begin{equation} \label{image}
  j^\infty \Big(\bigcup \Big\{\cE^{\{\fN\}}_{0,n} : \cE^{\{\fN\}}_{0,n} \text{ is quasianalytic} \Big\}\Big);
\end{equation} 
see Theorems \ref{matrixR} and \ref{matrixB}. 
Note that, since trivially $\cE^{(\fN)}_{0,n} \subseteq \cE^{\{\fN\}}_{0,n}$, 
this result implies all statements above. In particular, $\La^{\{\om\}}_n$ (resp.\ $\La^{(\om)}_n$) is not contained in
\[
  j^\infty \Big(\bigcup \Big\{\cE^{\{\si\}}_{0,n} : \si \text{ is a quasianalytic weight function} \Big\}\Big).   
\] 
Our proof is based on Bernstein's theorem on absolutely monotone 
functions \cite{Bernstein14} and on a theorem due to Bang \cite{Bang53} (Theorem~\ref{thmBang} below), 
which we recall with full proof for the sake of 
completeness. 

Let us emphasize that our proof also provides some partial information on the image \eqref{image}. If $n=1$ (for simplicity) 
then \eqref{image} cannot contain any strictly positive sequence $a = (a_k)$ unless $a$ defines a real analytic germ. 
Even for a single quasianalytic weight sequence $M$ it is generally 
not known how to identify the elements of $j^\infty \cE^{\{M\}}_{0,1}$ among those of $\La^{\{M\}}_1$.     

We wish to mention the recent paper by Sfouli \cite{Sfouli14} in which Carleman's result is obtained for quasianalytic local rings 
defined in an abstract way. This abstract definition includes stability under composition and differentiation. These rather restrictive 
properties (see e.g.\ \cite{RainerSchindl14} for a characterization of the former) are not required in our setting. 
Moreover, the approach of Sfouli yields Carleman's result only in dimension $n\ge 2$.

\section{Weight sequences and Denjoy--Carleman classes} \label{secDC}

Let us recall some basic facts on weight sequences and define Denjoy--Carleman classes and its germs.

\subsection{Denjoy--Carleman classes and its germs} \label{DenjoyCarleman}

Let $M=(M_k)_{k \in \N}$ be a positive sequence and let $U \subseteq \R^n$ be an open non-empty set. 
Then the set $\cE^{\{M\}}(U)$ of all $f \in C^\infty (U)$ such that for all compact $K \subseteq U$ there exists 
$\rh>0$ with   
\[
  \|f\|^M_{K,\rh} := \sup_{x \in K,~\al \in \N^n} \frac{|\p^\al f(x)|}{\rh^{|\al|} M_{|\al|}} < \infty,
\]
is called the \emph{Denjoy--Carleman class of Roumieu type} associated with $M$. 
It is endowed with the natural projective topology over $K$ and 
the inductive topology over $\rh \in \N$. Analogously, we define the \emph{Denjoy--Carleman class of Beurling type} $\cE^{(M)}(U)$
consisting of all $f \in C^\infty (U)$ such that for all compact $K \subseteq U$ and all 
$\rh>0$, $\|f\|^M_{K,\rh}<\infty$, and endow it with its natural Fr\'echet topology ($1/\rh \in \N$).

Let us define the spaces of germs at $0 \in \R^n$, 
\begin{align*}
   \cE^{\{M\}}_{0,n} &:= \on{ind}_{k \in \N} \cE^{\{M\}}((-\tfrac{1}{k},\tfrac{1}{k})^n) \\ 
   \cE^{(M)}_{0,n} &:= \on{ind}_{k \in \N} \cE^{(M)}((-\tfrac{1}{k},\tfrac{1}{k})^n). 
\end{align*} 
Finally we consider the sequence spaces 
\begin{align*}
   \La^{\{M\}}_{n} &:= \{a=(a_\al) \in \C^{\N^n}  : \E \rh >0 : |a|^M_\rh < \infty  \}, \\ 
   \La^{(M)}_{n} &:= \{a=(a_\al) \in \C^{\N^n}  : \A \rh >0 : |a|^M_\rh < \infty  \}, 
\end{align*}
where
\[
  |a|^M_{\rh} := \sup_{\al \in \N^n} \frac{|a_\al|}{\rh^{|\al|} M_{|\al|}}.
\]
Then $\La^{\{M\}}_{n}$ is an (LB)-space and $\La^{(M)}_{n}$ is a Fr\'echet space.

With the sequence $M_k = k!$ we recover the real analytic functions $\cE^{\{k!\}}(U) = C^\om(U)$ and restrictions of the entire 
functions $\cE^{(k!)}(U) = \cH(\C^n)$ if $U$ is connected. 
We denote by $\cO_{0,n}$ the ring of germs of real analytic functions at $0 \in \R^n$. 

By convention we write $\cE^{[M]}$ if we mean either $\cE^{\{M\}}$ or $\cE^{(M)}$, similarly $\La^{[M]}_{n}$ stands for 
$\La^{\{M\}}_{n}$ and $\La^{(M)}_{n}$, etc.

\subsection{Weight sequences and properties of Denjoy--Carleman classes} \label{weights}

We shall impose some mild regularity properties on the sequence $M=(M_k)$ that guarantee, in particular, 
that $\cE^{[M]}_{0,n}$ is a ring.

By definition, a \emph{weight sequence} is a sequence of positive real numbers $M=(M_k)_{k\in \N}$ such that:
\begin{align}
  &1 = M_0 \le M_1, \label{M1}\\
  &k \mapsto M_k \text{ is logarithmically convex (log-convex for short)}, \label{M2}\\
  &\liminf_{k} m_k^{1/k}>0. \label{M3}
\end{align}
Given a sequence $M=(M_k)$ we associate the sequences $m=(m_k)$ and $\mu=(\mu_k)$ given by
\begin{align*}
  m_k := \frac{M_k}{k!}, \quad
  \mu_k := \frac{M_k}{M_{k-1}}. 
\end{align*} 
Note that \eqref{M1} and \eqref{M2} imply that $M_k$ and $M_k^{1/k}$ are non-decreasing.

\begin{remark}  
  Under the assumption that $C^\om \subseteq \cE^{\{M\}}$ (resp.\ $C^\om \subseteq \cE^{(M)}$) which we shall always make,
  \eqref{M1}, \eqref{M2}, and \eqref{M3} are no restriction of generality for our problem, because one can change to the log-convex minorant $\underline M$ of $M$ 
  which describes the same function space: $\cE^{[\underline M]} = \cE^{[M]}$, 
  see \cite[Theorem 2.15]{RainerSchindl12}, whereas $\La^{[\underline M]} \subseteq  \La^{[M]}$.    

  In \cite{RainerSchindl12} and \cite{RainerSchindl14} we denoted by $M=(M_k)$ the sequence which here 
  is denoted by $m=(m_k)$. We deviate from our former convention for notational simplicity.  
\end{remark}

For arbitrary positive sequences $M =(M_k)$ and $N=(N_k)$ we define 
\begin{align*}
   M \preceq N \quad :\Leftrightarrow  \quad \E C,\rh >0 \A k : M_k \le C \rh^k N_k 
   \quad \Leftrightarrow  \quad \sup_k \Big(\frac{M_k}{N_k}\Big)^{1/k} < \infty
 \end{align*} 
and  
\begin{align*}
   M \lhd N \quad :\Leftrightarrow  \quad \A \rh > 0 \E C>0 \A k : M_k \le C \rh^k N_k 
   \quad \Leftrightarrow  \quad \lim_k \Big(\frac{M_k}{N_k}\Big)^{1/k} =0.
 \end{align*}
Then $M \preceq N$ implies $\cE^{[M]} \subseteq \cE^{[N]}$ and $\La^{[M]} \subseteq \La^{[N]}$, and $M \lhd N$ implies 
$\cE^{\{M\}} \subseteq \cE^{(N)}$ and $\La^{\{M\}} \subseteq \La^{(N)}$. 
The converse implications hold if $M$ is a weight sequence; cf.\ \cite[Proposition 2.12]{RainerSchindl12} and \cite[Lemma 2.2]{KMRu}, 
that $\La^{(M)} \subseteq \La^{(N)}$ implies $M \preceq N$ follows from the argument in \cite{Bruna80/81}. 
In particular, \eqref{M3} holds if and only if the real analytic class is contained in $\cE^{\{M\}}$, and furthermore, if and only if 
the restrictions of all entire functions are contained in $\cE^{(M)}$. The inclusion of the real analytic class in $\cE^{(M)}$ 
is equivalent to the condition
\begin{align}
  m_k^{1/k} \to \infty. \label{M4}
\end{align}

A weight sequence $M=(M_k)$ is called \emph{quasianalytic} if 
\[
  \sum_{k=1}^\infty \frac{1}{\mu_k} = \infty, \quad \text{ or equivalently, } \quad \sum_{k=1}^\infty \frac{1}{M_k^{1/k}} = \infty.
\]
The famous Denjoy--Carleman theorem (cf.\ \cite[Theorem 2.1]{KMRu}) 
holds that $M=(M_k)$ is quasianalytic if and only if $\cE^{[M]}$ is \emph{quasianalytic}, i.e.,
for open connected $U\subseteq \R^n$ and each $a \in U$ the Borel mapping $f \mapsto (\p^\al f(a))_\al$ is injective on $\cE^{[M]}(U)$. 

\begin{remark} \label{bump}
  A class $\cE^{[M]}$ is called non-quasianalytic if it is not quasianalytic. This is equivalent to the fact that 
  there exist non-trivial $\cE^{[M]}$-functions with compact support.
\end{remark}

\section{A proof of Carleman's theorem}

We have the Borel mapping 
\begin{equation} \label{Borel}
  j^\infty : \cE^{[M]}_{0,n} \to \La^{[M]}_n, \quad f \mapsto (\p^\al f(0))_{\al \in \N^n}.
\end{equation}
If $M=(M_k)$ is a quasianalytic weight sequence, then this mapping is injective. 
In this section we will show that it is never surjective if $\cO_{0,n} \subsetneq \cE^{[M]}_{0,n}$.

\subsection{The Roumieu case}

Let us first concentrate on the Roumieu case.
Due to a theorem of Carleman, the mapping $j^\infty : \cE^{\{M\}}_{0,n} \to \La^{\{M\}}_n$ is 
never surjective if $\cO_{0,n} \subsetneq \cE^{\{M\}}_{0,n}$, or equivalently,  
\begin{equation*}
  \sup_k m_k^{1/k} = \infty.
\end{equation*}
A concise proof (following the main ideas of Carleman) may be found in \cite{Thilliez08}. 

We shall give another proof based on Bernstein's theorem (cf.\ \cite[p.~146]{Widder41} for a proof and \cite{Boas71} for a survey of related results) 
and the following elementary theorem 
due to Bang \cite{Bang53}. 
We reproduce the proof of the latter for the convenience of the reader and for the sake of completeness 
(cf.\ also \cite{NazarovSodinVolberg04}).

\begin{theorem}[Bang \cite{Bang53}] \label{thmBang}
  Let $M=(M_k)$ be a quasianalytic weight sequence and let $f \in C^\infty([0,1])$ satisfy
  \begin{equation} \label{assumption}
    \sup_{x \in [0,1]} |f^{(j)}(x)| \le M_j, \quad j \in \N. 
  \end{equation}
  If $f$ is not identically $0$ and for all $j \in \N$ there exists 
  $x_j \in [0,1]$ such that $f^{(j)}(x_j) = 0$, then the series $\sum_{j=0}^\infty |x_j-x_{j+1}|$ is divergent. 
\end{theorem}

\begin{proof}
  For $N\in \N$ and $x \in [0,1]$ set  
  \[
    B_{f,N}(x) :=  \max _{j \ge N} \frac{|f^{(j)}(x)|}{e^j M_j} .
  \]
  Let us collect some properties of $B_{f,N}$:
  \begin{enumerate}
    \item $B_{f,N}(x) \le e^{-N}$,
    \item $B_{f,N}(x) \ge B_{f,N+1}(x)$, and $B_{f,N}(x) = B_{f,N+1}(x)$ if $f^{(N)}(x) = 0$,
    \item for all $k > N$ and all $x, x+h \in [0,1]$ $(h \ne0)$,  
    \[
      B_{f,N}(x+h) < \max\{B_{f,N}(x), e^{-k}\}\, e^{e |h| \mu_k}. 
    \]
  \end{enumerate} 
  (1) and (2) follow easily from the definition and from \eqref{assumption}. To see (3)
  let $k > N$, $N \le j < k$, and $x, x+h \in [0,1]$. Then, by Taylor's formula, for some $\xi$ between $x$ and $x+h$,
  \begin{align*}
    \frac{|f^{(j)}(x+h)|}{e^j M_j} 
    &\le \sum_{i = 0}^{k-j-1} \frac{|f^{(j+i)}(x)|\, |h|^i}{e^j M_j\, i!} 
    + \frac{ |f^{(k)}(\xi)|\, |h|^{k-j}}{e^j M_j\, (k-j)!}  \\
    &= \sum_{i = 0}^{k-j-1} \frac{M_{j+i}}{M_j}\frac{|f^{(j+i)}(x)| }{e^{j+i} M_{j+i}} \frac{(e|h|)^i}{i!} 
    + e^{-k} \frac{M_{k}}{M_j} \frac{ |f^{(k)}(\xi)|}{M_k} \frac{ (e|h|)^{k-j}}{(k-j)!} \\
    &\le B_{f,N}(x) \sum_{i = 0}^{k-j-1} \Big(\frac{M_{k}}{M_{k-1}}\Big)^i \frac{(e|h|)^i}{i!} 
    + e^{-k} \Big(\frac{M_{k}}{M_{k-1}}\Big)^{k-j}  \frac{ (e|h|)^{k-j}}{(k-j)!} \\
    &< \max\{B_{f,N}(x), e^{-k}\}\, e^{e |h| \mu_k},
  \end{align*}
  where we used that $M=(M_k)$ is log-convex.
  If $j \ge k$ then, by \eqref{assumption}, 
  \begin{align*}
    \frac{|f^{(j)}(x+h)|}{e^j M_j} \le e^{-j} < \max\{B_{f,N}(x), e^{-k}\}\, e^{e |h| \mu_k}.
  \end{align*}
  This implies (3).

  Let $f$ and $x_j$ be as in the theorem. 
  Set $\ta_k := \sum_{j=0}^{k-1} |x_j-x_{j+1}|$, $k \ge 1$, $\ta_0:= 0$, and define for $t \in [\ta_{N-1},\ta_N]$, 
  \[
    \tilde B_{f,N} (t) := 
    \begin{cases}
       B_{f,N}(x_{N-1} +\ta_{N-1} -t) & \text{if } x_N < x_{N-1}, \\ 
       B_{f,N}(x_{N-1} -\ta_{N-1} +t) & \text{if } x_N \ge x_{N-1}.
    \end{cases}
  \] 
  By (3), the function $\tilde B_{f,N}$ is continuous and, by (2), $\tilde B_{f,N}(\ta_N) = B_{f,N}(x_N) = B_{f,N+1}(x_N) =\tilde B_{f,N+1}(\ta_N)$. 
  So we obtain a continuous function $\tilde B_{f}$ on the interval $[0,\ta)$, where $\ta:=\sup_k \ta_k$, by setting 
  \[
    \tilde B_{f}(t) := \tilde B_{f,N}(t) \quad \text{ if } t \in [\ta_{N-1},\ta_N],~ N \ge 1.
  \]
  By (1) and (2) we find that $\tilde B_{f}(t) \le e^{-N}$ for all $t \ge \ta_{N-1}$ and hence $\tilde B_f(t) \to 0$ as $t \to \ta$. 
  Since $f$ and thus also $\tilde B_f$ does not vanish identically, 
  the range of $\tilde B_f$ contains all numbers $e^{-k}$ for sufficiently large $k$, say $k \ge k_0$. 
  So we may choose a strictly increasing sequence $t_k$ such that $\tilde B_f(t_k) = e^{-k}$ and $\tilde B_f(t) > e^{-k}$ for all 
  $t \in (t_{k-1},t_k)$ (recursively, take for $t_k$ the smallest $t \in \tilde B_f^{-1}(e^{-k})$ with $t>t_{k-1}$). 
  By (3) (applied to each interval in the subdivision of $(t_{k-1},t_k)$ induced by the points $\ta_N$ between $t_{k-1}$ and $t_k$) 
  we may conclude that 
  \[
    \tilde B_{f}(t_{k-1}) \le \tilde B_{f}(t_k)\, e^{e (t_k-t_{k-1}) \mu_k},
  \]
  or equivalently, 
  \[
     t_k-t_{k-1} \ge \frac1{e \mu_k},
  \]
  and therefore
  \[
    t_k \ge t_{k_0} + \frac1{e} \sum_{j=k_0+1}^k \frac1{\mu_j}.
  \]
  By the choice of the sequence $t_k$ we find that $\ta_k \ge t_k$, which implies the assertion as $M=(M_k)$ is quasianalytic.
\end{proof}

\begin{corollary}[Bang \cite{Bang53}] \label{corBang}
  Let $M=(M_k)$ be a quasianalytic weight sequence and let $f \in C^\infty([0,1])$ satisfy \eqref{assumption}.
  If $f^{(j)}(0) > 0$ for all $j \in \N$, then $f^{(j)}(x) > 0$ for all $x \in [0,1]$ and $j \in \N$.
\end{corollary}

\begin{proof}
  Suppose that $f^{(j)}(0) > 0$ for all $j \in \N$ and that some  
  derivative $f^{(j)}$ has a zero $x_{j} \in (0,1]$. By Rolle's theorem, we find a strictly decreasing sequence 
  $x_j > x_{j+1} > \cdots > 0$, where $x_k$ is a zero of $f^{(k)}$ for all $k \ge j$. This contradicts Theorem \ref{thmBang}.
\end{proof}

We may deduce not only that $j^\infty : \cE^{\{M\}}_{0,n} \to \La^{\{M\}}_n$ is not surjective if 
$\cO_{0,n} \subsetneq \cE^{\{M\}}_{0,n}$, 
but that there exist elements in $\La^{\{M\}}_n$ that are not contained in $j^\infty  \cE^{\{N\}}_{0,n}$ for 
\emph{any} quasianalytic weight sequence $N = (N_k)$.

\begin{theorem} \label{seqRoumieu}
  Let $M=(M_k)$ be a quasianalytic weight sequence such that $\cO_{0,n} \subsetneq \cE^{\{M\}}_{0,n}$. 
  Then there exist elements in $\La^{\{M\}}_n$ that are not contained in $j^\infty  \cE^{\{N\}}_{0,n}$ for 
  any quasianalytic weight sequence $N = (N_k)$.   
\end{theorem}

\begin{proof}
  Without loss of generality we may assume that $n=1$.
  Let $a=(a_j) \in \La^{\{M\}}_1$ be positive, i.e., $a_j >0$ for all $j$. 
  Let $N =(N_k)$ be any quasianalytic weight sequence. 
  We claim that 
  if there exists $f \in \cE^{\{N\}}_{0,1}$ such that $j^\infty f = a$ then $f \in \cO_{0,1}$.  
  There is $r>0$ such that $f \in \cE^{\{N\}}((-r,r))$ and $0< r_1 <r$ and $\rh>0$ such that 
  \begin{equation*}
    \sup_{x \in [0,r_1]} |f^{(j)}(x)| \le \rh^{j+1} N_j, \quad j \in \N; 
  \end{equation*}
  abusing notation we denote germs and its representatives by the same symbol.
  Let us define $\tilde f (x) := \rh^{-1} f(r_1 x)$ and $\tilde N_j := (\rh\, r_1)^j N_j$. Then  
  \begin{equation*}
    \sup_{x \in [0,1]} |\tilde f^{(j)}(x)| \le \tilde N_j, \quad j \in \N, 
  \end{equation*}
  and hence Corollary \ref{corBang} implies that $\tilde f^{(j)}(x) >0$ for all $x \in [0,1]$ and all $j \in \N$, 
  that is $f^{(j)}(x) >0$ for all $x \in [0,r_1]$ and all $j \in \N$. 
  By Bernstein's theorem (e.g.\ \cite[p.~146]{Widder41}), $f \in \cO_{0,1}$. 

  Thus if $a=(a_j)$ is chosen such that it does not define a real analytic germ, which is possible by the assumption 
  $\cO_{0,n} \subsetneq \cE^{\{M\}}_{0,n}$, then it cannot belong to $j^\infty  \cE^{\{N\}}_{0,1}$ for any 
  quasianalytic weight sequence $N = (N_k)$.
\end{proof}

\subsection{The Beurling case} 

Here we assume that $\cO_{0,n} \subsetneq \cE^{(M)}_{0,n}$ which is equivalent to the condition \eqref{M4}, i.e., 
$m_k^{1/k} \to \infty$.
We shall use the following representation result which is a special case of Proposition \ref{Beurlingrepresentation} below.

\begin{proposition}[{\cite[Proposition 2.12]{RainerSchindl12}}] \label{repBeurling}
  If $M=(M_k)$ is a positive sequence such that $m_k^{1/k} \to \infty$, then
  \begin{align*}
    \La^{(M)}_n = \bigcup \big\{\La^{\{L\}}_n : L \lhd M,~ \ell_k^{1/k} \to \infty\big\}. 
  \end{align*}
\end{proposition}

This proposition allows us to reduce the Beurling to the Roumieu case.

\begin{theorem}
  Let $M=(M_k)$ be a quasianalytic weight sequence such that $\cO_{0,n} \subsetneq \cE^{(M)}_{0,n}$. 
  Then there exist elements in $\La^{(M)}_n$ that are not contained in $j^\infty  \cE^{\{N\}}_{0,n}$ for 
  any quasianalytic weight sequence $N = (N_k)$.  
\end{theorem}

In particular, there are elements in $\La^{(M)}_n$ not contained in 
$j^\infty  \cE^{(N)}_{0,n}$ for any quasianalytic weight sequence $N = (N_k)$, 
since always $\cE^{(N)} \subseteq \cE^{\{N\}}$.

\begin{proof}
  Let $L=(L_k)$ be a positive sequence satisfying $L \lhd M$ and $\ell_k^{1/k} \to \infty$. 
  Let $\underline L=(\underline L_k)$ denote the log-convex minorant of $L$. We still have $\underline L \lhd M$ 
  and $\underline \ell_k^{1/k} \to \infty$; cf.\ \cite[Lemma 2.6 and Theorem 2.15]{RainerSchindl12}. 
  Thus $\underline L$ is a quasianalytic weight sequence (since so is $M$). The condition $\underline \ell_k^{1/k} \to \infty$
  implies that $\cO_{0,n} \subsetneq \cE^{\{\underline L\}}_{0,n}$. 

  By Theorem \ref{seqRoumieu}, there exist elements in $\La^{\{\underline L\}}_n$ 
  that are not in $j^\infty \cE^{\{N\}}_{0,n}$ for any quasianalytic weight sequence $N=(N_k)$. 
  This implies the statement by Proposition~\ref{repBeurling}. 
\end{proof}

\section{Weight functions, weight matrices,\\ and Braun--Meise--Taylor classes}

\subsection{Weight functions} \label{weightfunction}

A \emph{weight function} is a continuous increasing function $\om : [0,\infty) \to [0,\infty)$ with $\om|_{[0,1]} =0$ and 
$\lim_{t \to \infty} \om(t) = \infty$ that satisfies
\begin{align}
   & \om(2t) = O(\om(t)) \quad\text{ as } t \to \infty, \label{om1}\\
   & \om(t) = O(t) \quad\text{ as } t \to \infty, \label{om2}\\
   & \log t = o(\om(t)) \quad\text{ as } t \to \infty, \label{om3}\\
   & \vh(t) := \om(e^t) \text{ is convex}.  \label{om4}
\end{align}

For a weight function $\om$ we consider the \emph{Young conjugate} $\vh^*$ of $\vh$,
\[
  \vh^*(x) := \sup_{y\ge 0} xy-\vh(y), \quad x \ge 0,
\] 
which is a convex increasing function satisfying $\vh^*(0)=0$, $\vh^{**}=\vh$, and $x/\vh^*(x) \to 0$ as $x \to \infty$; 
see \cite{BMT90}.

\subsection{Braun--Meise--Taylor classes and its germs} \label{BMT}

Let $\om$ be a weight function and let $U \subseteq \R^n$ be an open non-empty set. 
Then the set $\cE^{\{\om\}}(U)$ of all $f \in C^\infty (U)$ such that for all compact $K \subseteq U$ there exists 
$\rh>0$ with   
\[
  \|f\|^\om_{K,\rh} := \sup_{x \in K,~\al \in \N^n} \frac{|\p^\al f(x)|}{\exp(\frac{1}{\rh} \vh^*(\rh |\al|))} < \infty,
\]
is called the \emph{Braun--Meise--Taylor class of Roumieu type} associated with $\om$. 
It is endowed with the natural projective topology over $K$ and 
the inductive topology over $\rh \in \N$. Analogously, we define the \emph{Braun--Meise--Taylor class of Beurling type} $\cE^{(\om)}(U)$
consisting of all $f \in C^\infty (U)$ such that for all compact $K \subseteq U$ and all 
$\rh>0$, $\|f\|^\om_{K,\rh}<\infty$, and endow it with its natural Fr\'echet topology ($1/\rh \in \N$).    

Let us define the rings of germs at $0 \in \R^n$,
\begin{align*}
   \cE^{\{\om\}}_{0,n} &:= \on{ind}_{k \in \N} \cE^{\{\om\}}((-\tfrac{1}{k},\tfrac{1}{k})^n) \\ 
   \cE^{(\om)}_{0,n} &:= \on{ind}_{k \in \N} \cE^{(\om)}((-\tfrac{1}{k},\tfrac{1}{k})^n). 
\end{align*} 
and consider the sequence spaces 
\begin{align*}
   \La^{\{\om\}}_{n} &:= \{a=(a_\al) \in \C^{\N^n}  : \E \rh >0 : |a|^\om_\rh < \infty  \}, \\ 
   \La^{(\om)}_{n} &:= \{a=(a_\al) \in \C^{\N^n}  : \A \rh >0 : |a|^\om_\rh < \infty  \}, 
\end{align*}
where
\[
  |a|^\om_{\rh} := \sup_{\al \in \N^n} \frac{|a_\al|}{\exp(\frac{1}{\rh} \vh^*(\rh |\al|))}.
\]
Then $\La^{\{\om\}}_{n}$ is an (LB)-space and $\La^{(\om)}_{n}$ is a Fr\'echet space.

With $\om(t) = t$ we recover the real analytic functions $\cE^{\{t\}}(U) = C^\om(U)$ and restrictions of the entire 
functions $\cE^{(t)}(U) = \cH(\C^n)$ if $U$ is connected. 

Again $\cE^{[\om]}$ stands for either $\cE^{\{\om\}}$ or $\cE^{(\om)}$, $\La^{[\om]}_n$ for $\La^{\{\om\}}_n$ or $\La^{(\om)}_n$, etc.

For weight functions $\om$ and $\si$ we define 
\[
  \om \preceq \si \quad :\Leftrightarrow \quad \si(t) = O(\om(t)) \text{ as } t \to \infty
\]
and 
\[
  \om \lhd \si \quad :\Leftrightarrow \quad \si(t) = o(\om(t)) \text{ as } t \to \infty.
\]
Then $\om \preceq \si$ if and only if $\cE^{[\om]} \subseteq \cE^{[\si]}$ if and only if $\La^{[\om]} \subseteq \La^{[\si]}$, 
and  
$\om \lhd \si$ if and only if 
$\cE^{\{\om\}} \subseteq \cE^{(\si)}$ if and only if $\La^{\{\om\}} \subseteq \La^{(\si)}$; 
cf.\ \cite[Corollary 5.17]{RainerSchindl12}.
In particular, \eqref{om2} holds if and only if the real analytic class is contained in $\cE^{\{\om\}}$, and furthermore, if and only if 
the restrictions of all entire functions are contained in $\cE^{(\om)}$. The inclusion of the real analytic class in $\cE^{(\om)}$ 
is equivalent to the condition
\begin{equation*}
  \om(t) = o(t) \quad \text{ as } t \to \infty. \label{om5}
\end{equation*}

A weight function $\om$ is called \emph{quasianalytic} if 
\[
  \int_1^\infty \frac{\om(t)}{t^2} \,dt =\infty.
\]
This condition is equivalent to quasianalyticity of $\cE^{[\om]}$.

\subsection{The weight matrix associated with a weight function}

Given a weight function $\om$ we may associate a \emph{weight matrix} $\fW = \{W^x\}_{x>0}$ by setting 
\begin{equation*}
  W^x_k := \exp(\tfrac{1}{x}\vh^*(x k)), \quad k \in \N.
\end{equation*}
By the properties of $\vh^*$, each $W^x$ is a weight sequence (in the sense of Section \ref{weights}) and $W^x \le W^y$ if $x \le y$.

The weight function $\om$ is quasianalytic if and only if each (equivalently, some) $W^x$ is quasianalytic; 
see \cite[Corollary 5.8]{RainerSchindl12}. 

The associated weight matrix $\fW$ allows us to describe any Braun--Meise--Taylor class $\cE^{[\om]}$ as 
a union or an intersection of Denjoy--Carleman classes.

\begin{theorem}[{\cite[Corollary 5.15]{RainerSchindl12}}] \label{representation}
  Let $\om$ be a weight function and let $\fW = \{W^x\}_{x>0}$ be the associated weight matrix. 
  Let $U \subseteq \R^n$ be any open non-empty set.
  Then
  \begin{align*}
    \cE^{\{\om\}}(U) &= \on{proj}_{K \subseteq U} \on{ind}_{x > 0} \on{ind}_{\rh>0} \cE^{W^x}_\rh(K), \\
    \cE^{(\om)}(U) &= \on{proj}_{K \subseteq U} \on{proj}_{x > 0} \on{proj}_{\rh>0} \cE^{W^x}_\rh(K) 
  \end{align*}
  as locally convex spaces ($K$ runs through a compact exhaustion of $U$). Here $\cE^{W^x}_\rh(K)$ denotes the Banach space
  \[
    \cE^{W^x}_\rh(K) := \{f \in C^\infty(K) : \|f\|^{W^x}_{K,\rh}<\infty\}.
  \]
\end{theorem}

\subsection{Weight matrices and associated ultradifferentiable classes}

More abstractly, we define a \emph{weight matrix} to be a  
family of weight sequences $\fM = \{M^x\}_{x \in X}$ indexed by a subset $X \subseteq \R$ such that 
\begin{equation} \label{order}
  M^x \le M^y \quad \text{if }  x \le y.  
\end{equation}

For a weight matrix $\fM = \{M^x\}_{x \in X}$ and an open non-empty set $U \subseteq \R^n$, we define the 
locally convex spaces 
\begin{align*}
    \cE^{\{\fM\}}(U) &= \on{proj}_{K \subseteq U} \on{ind}_{x > 0} \on{ind}_{\rh>0} \cE^{M^x}_\rh(K), \\
    \cE^{(\fM)}(U) &= \on{proj}_{K \subseteq U} \on{proj}_{x > 0} \on{proj}_{\rh>0} \cE^{M^x}_\rh(K),
\end{align*}
its rings of germs at $0 \in \R^n$,
\begin{align*}
   \cE^{\{\fM\}}_{0,n} &:= \on{ind}_{k \in \N} \cE^{\{\fM\}}((-\tfrac{1}{k},\tfrac{1}{k})^n), \\ 
   \cE^{(\fM)}_{0,n} &:= \on{ind}_{k \in \N} \cE^{(\fM)}((-\tfrac{1}{k},\tfrac{1}{k})^n), 
\end{align*} 
and the sequence spaces 
\begin{align*}
   \La^{\{\fM\}}_{n} &:= \{a=(a_\al) \in \C^{\N^n}  : \E x \in X \E \rh >0 : |a|^{M^x}_\rh < \infty  \}, \\ 
   \La^{(\fM)}_{n} &:= \{a=(a_\al) \in \C^{\N^n}  : \A x \in X \A \rh >0 : |a|^{M^x}_\rh < \infty  \}
\end{align*}
with the natural (LB)- and Fr\'echet topology. As usual $\cE^{[\fM]}$ means either $\cE^{\{\fM\}}$ or $\cE^{(\fM)}$, etc.

For weight matrices $\fM=\{M^x\}_{x \in X}$ and $\fN = \{N^y\}_{y \in Y}$ we define 
\begin{align*}
  \fM \{\preceq\} \fN \quad &: \Leftrightarrow \quad \A x \in X \E y \in Y : M^x \preceq N^y \\
  \fM (\preceq) \fN \quad &: \Leftrightarrow \quad \A y \in Y \E x \in X  : M^x \preceq N^y
\end{align*}
and 
\begin{align*}
  \fM \{\lhd) \fN \quad &: \Leftrightarrow \quad \A x \in X \A y \in Y : M^x \lhd N^y.
\end{align*}
Then $\fM [\preceq] \fN$ if and only if $\cE^{[\fM]} \subseteq \cE^{[\fN]}$  if and only if $\La^{[\fM]} \subseteq \La^{[\fN]}$, 
and $\fM \{\lhd) \fN$ if and only if $\cE^{\{\fM\}} \subseteq \cE^{(\fN)}$  if and only if $\La^{\{\fM\}} \subseteq \La^{(\fN)}$; 
see \cite[Proposition 4.6]{RainerSchindl12}.

Analogously to Theorem \ref{representation} we get:

\begin{proposition} \label{repLa}
  Let $\om$ be a weight function and let $\fW = \{W^x\}_{x>0}$ be the associated weight matrix.
  Then
  \begin{align*}
    \La^{\{\om\}}_n = \La^{\{\fW\}}_{n} \quad \text{and} \quad   
    \La^{(\om)}_n = \La^{(\fW)}_{n}  
  \end{align*}
  as locally convex spaces. 
\end{proposition}

\begin{proof}
  This follows from the proof of Theorem \ref{representation} in \cite[Theorem 5.14]{RainerSchindl12}. 
  The argument is based on the following two facts: by definition,
  \begin{equation} \label{5.00}
    |a|^\om_x = |a|^{W^x}_1, \quad a \in \C^{\N^n},
  \end{equation}
  and, by \cite[Lemma 5.9]{RainerSchindl12}, 
  \begin{equation} \label{5.10}
    \A \si>0 \E H\ge 1 \A x >0 \E C \ge 1 \A k \in \N : \si^k W^x_k \le C W^{Hx}_k.
  \end{equation}
  The (continuous) inclusions $\La^{\{\om\}}_n \subseteq \La^{\{\fW\}}_{n}$ and    
    $\La^{(\om)}_n \supseteq \La^{(\fW)}_{n}$ follow easily from \eqref{5.00}.  

  If we combine \eqref{5.00} and \eqref{5.10} we obtain
  \begin{align*} 
   \A \si>0 &\E H\ge 1 \A x >0 \E C \ge 1 : \\ 
   & |a|^\om_{Hx} \le C\, |a|^{W^x}_\si  ~\text{ and }~ |a|^{W^x}_{1/\si} \le  C\, |a|^\om_{x/H} , \quad a \in \C^{\N^n},
  \end{align*}
  which implies the continuous inclusions
     $\La^{\{W^x\}}_{n} \subseteq \La^{\{\om\}}_n$ and   
    $\La^{(\om)}_n \subseteq \La^{(W^x)}_{n}$,  
  for all $x >0$.  
\end{proof}

\section{Non-surjectivity of the Borel mapping\\ for proper quasianalytic classes}

We shall show in this section that the Borel mapping is never surjective in the proper quasianalytic setting.
We will work in the framework of ultradifferentiable classes $\cE^{[\fM]}$ defined in terms of a weight matrix $\fM$. 
In view of Theorem \ref{representation} this includes all Braun--Meise--Taylor classes and thus we recover the result 
of Bonet and Meise \cite{BonetMeise13}. 
The approach via weight matrices allows us to apply the results on Denjoy--Carleman classes in Section \ref{secDC} 
in a direct way.

Let $\fM=\{M^x\}_{x \in X}$ be a weight matrix.
Let us consider the Borel mapping 
\begin{equation} \label{Borel2}
  j^\infty : \cE^{[\fM]}_{0,n} \to \La^{[\fM]}_n, \quad f \mapsto (\p^\al f(0))_{\al \in \N^n}.
\end{equation}
The mapping \eqref{Borel2} specializes to the mapping \eqref{Borel} if $\fM$ consists of a single weight sequence $M$,
and it specializes to the mapping
\begin{equation*} 
  j^\infty : \cE^{[\om]}_{0,n} \to \La^{[\om]}_n, \quad f \mapsto (\p^\al f(0))_{\al \in \N^n},
\end{equation*}
if $\om$ is a weight function, thanks to Theorem \ref{representation} and Proposition \ref{repLa}.

\subsection{Quasianalytic weight matrices}

It is easy to see that
the ring $\cE^{\{\fM\}}_{0,n}$ is quasianalytic (i.e., the Borel mapping $j^\infty : \cE^{\{\fM\}}_{0,n} \to \La^{\{\fM\}}_n$ 
is injective) if and only if each weight sequence $M^x$ in the weight matrix $\fM=\{M^x\}_{x \in X}$
is quasianalytic. 

In the Beurling case, $\cE^{(\fM)}_{0,n}$ is quasianalytic if and only if at least one weight sequence $M^x$ in the weight matrix 
$\fM=\{M^x\}_{x \in X}$
is quasianalytic; this follows from \cite[Proposition 4.7]{Schindl15}. In that case we can assume that all weight sequences in $\fM$ are quasianalytic
by removing all non-quasianalytic ones; by the property \eqref{order} this leaves the spaces $\cE^{(\fM)}(U)$, $\cE^{(\fM)}_{0,n}$, 
and $\La^{(\fM)}_n$ unchanged.

In light of this remark we call a weight matrix $\fM=\{M^x\}_{x \in X}$ \emph{quasianalytic} if each weight sequence $M^x$ is 
quasianalytic. (We warn the reader that the formal negation of this notion, i.e., \emph{$\fM$ is not quasianalytic}, means that $\cE^{\{\fM\}}_{0,n}$ 
is non-quasianalytic,
but not necessarily $\cE^{(\fM)}_{0,n}$.)

\subsection{The Roumieu case}

We shall assume that $\fM=\{M^x\}_{x \in X}$ is a quasianalytic weight matrix such that $\cO_{0,n} \subsetneq \cE^{\{\fM\}}_{0,n}$. 
The latter condition holds if and only if $\cO_{0,n} \subsetneq \cE^{\{M^x\}}_{0,n}$ for some $x \in X$, or equivalently
\begin{equation*}
  \E x \in X : \sup_k (m_k^x)^{1/k} = \infty, 
\end{equation*}
where $m_k^x := M_k^x/k!$.

\begin{theorem} \label{matrixR}
  Let $\fM=\{M^x\}_{x \in X}$ be a quasianalytic weight matrix such that $\cO_{0,n} \subsetneq \cE^{\{\fM\}}_{0,n}$. 
  Then there exist elements in $\La^{\{\fM\}}_n$ that are not contained in $j^\infty \cE^{\{\fN\}}_{0,n}$ for any quasianalytic 
  weight matrix $\fN = \{N^y\}_{y \in Y}$.  
\end{theorem}

\begin{proof}
  By assumption there exists $x \in X$ such that $\cO_{0,n} \subsetneq \cE^{\{M^x\}}_{0,n}$. Then Theorem~\ref{seqRoumieu} 
  implies that there is an element $a = (a_\al) \in \La^{\{M^x\}}_n \subseteq \La^{\{\fM\}}_n$ such that $a \not\in j^\infty \cE^{\{N\}}_{0,n}$ 
  for all quasianalytic weight sequences $N=(N_k)$.  
  
  In particular, $a \not\in j^\infty \cE^{\{\fN\}}_{0,n}$ for every quasianalytic weight matrix 
  $\fN = \{N^y\}_{y \in Y}$.  
  In fact,
  suppose that $a \in j^\infty \cE^{\{\fN\}}_{0,n}$ for some quasianalytic weight matrix $\fN$. 
  Then there exist $r>0$ and $f \in \cE^{\{\fN\}}((-2r,2r)^n)$ such that 
  $j^\infty f = a$. By restriction, we can assume that $f \in \cE^{\{\fN\}}([-r,r]^n)$ and in turn that there exists $y \in Y$ 
  such that $f \in \cE^{\{N^y\}}([-r,r]^n)$.  
  But this contradicts the first paragraph.
\end{proof}

In view of Theorem \ref{representation} and Proposition \ref{repLa} we immediately obtain the following corollary.

\begin{corollary}
  Let $\om$ be a quasianalytic weight function such that $\cO_{0,n} \subsetneq \cE^{\{\om\}}_{0,n}$. 
  Then there exist elements in $\La^{\{\om\}}_n$ that are not contained in $j^\infty \cE^{\{\si\}}_{0,n}$ for any quasianalytic 
  weight function $\si$.  
\end{corollary}

Note that the strict inclusion $\cO_{0,n} \subsetneq \cE^{\{\om\}}_{0,n}$ holds if and only if 
\begin{equation*}
 \liminf_{t \to \infty} \frac{\om(t)}{t} = 0   
\end{equation*} 
which is immediate from the inclusion relations recalled in Section \ref{BMT}.

\subsection{The Beurling case}

Here we assume that $\fM=\{M^x\}_{x \in X}$ is a quasianalytic weight matrix such that $\cO_{0,n} \subsetneq \cE^{(\fM)}_{0,n}$. 
As we will see below this strict inclusion holds if and only if $\cO_{0,n} \subsetneq \cE^{(M^x)}_{0,n}$ for all $x \in X$, or equivalently
\begin{equation} \label{Bcond}
  \A x \in X : (m_k^x)^{1/k} \to \infty, 
\end{equation}
where $m_k^x := M_k^x/k!$.

We will reduce the Beurling to the Roumieu case. The key to this reduction is the following lemma. 

\begin{lemma} \label{LNM}
  Let $\fM=\{M^x\}_{x \in X}$ be a weight matrix. 
  Let $L=(L_k)$ be any positive sequence satisfying $L \{\lhd) \fM$. Then there exists a positive sequence $N = (N_k)$ 
  satisfying $L \lhd N \{\lhd) \fM$.
\end{lemma}

\begin{proof}
  Without loss of generality we can assume that $X=(0,\infty)$. The assumption $L \{\lhd) \fM$  precisely means that
  \begin{equation*}
    \A x \in (0,\infty) \A \rh >0 \E C \ge 1 \A k \in \N : 
    L_k \le C\, \rh^k M^x_k.
  \end{equation*}
  In particular, (taking $x=\rh=1/p$)
  \begin{equation} \label{diagonal}
    \A p \in \N_{\ge1} \E C \ge 1 \A k \in \N : 
    L_k \le C \,p^{-k} M^{1/p}_k.
  \end{equation}
  Let $C_p$ denote the minimal constant $C$ such that \eqref{diagonal} holds. This defines a non-decreasing sequence $(C_p)_p$ (by \eqref{order}). 
  Fix a real number $A > 1$. 
  Choose a strictly increasing sequence $(j_p)_{p\ge 1}$ of positive integers such that $C_p \le A^{j_p}$.

  We define 
  \begin{equation*}
    N_j := 
      \sqrt{L_j  M^{1/p}_j} \quad  \text{ for }\quad j_p \le j < j_{p+1};    
  \end{equation*}
  for $0 \le j < j_1$ any choice of $N_j$ works. 
  Then $L \lhd N$ since, by \eqref{diagonal}, for $j_p \le j < j_{p+1}$, 
  \begin{align*}
    \Big(\frac{L_j}{N_j}\Big)^{1/j} = \Big(\sqrt{\frac{L_j}{M^{1/p}_j}}\Big)^{1/j} \le \sqrt{\frac{C_p^{1/j}}{p}} \le \sqrt{\frac{A}{p}}  
  \end{align*}
  which tends to 0 as $j \to \infty$. Let $x >0$ be fixed. If $j_p \le j < j_{p+1}$ where $p \ge 1/x$, then by \eqref{order}, 
  \begin{align*}
    \Big(\frac{N_j}{M^x_j}\Big)^{1/j} = \Big(\frac{\sqrt{L_j  M^{1/p}_j}}{M^x_j}\Big)^{1/j} \le \Big(\sqrt{\frac{L_j}{M^x_j}}\Big)^{1/j} 
  \end{align*}
  which tends to 0 as $j \to \infty$ because $L \lhd M^x$. That is $N \{\lhd) \fM$ and the proof is complete.
\end{proof}

\begin{corollary} \label{corLNM}
  If $\fM=\{M^x\}_{x \in X}$ is a weight matrix satisfying \eqref{Bcond}, then there exists a positive sequence $N=(N_k) = (k!\, n_k)$ 
  satisfying $n_k^{1/k} \to 0$ 
  and $N \{\lhd) \fM$.
\end{corollary}

\begin{proof}
  If $L$ denotes the sequence $L_k=k!$, then \eqref{Bcond} means exactly $L \{\lhd) \fM$ and the above lemma implies the 
  assertion.
\end{proof}

Thus, if a weight matrix $\fM=\{M^x\}_{x \in X}$ satisfies \eqref{Bcond}, then there is a positive sequence $N=(N_k)$ such that  
$\cO_{0,n} \subsetneq \cE^{(N)}_{0,n} \subsetneq \cE^{(\fM)}_{0,n}$, and the assertion at the 
beginning of the section is proved.

\begin{proposition} \label{Beurlingrepresentation}
  If $\fM=\{M^x\}_{x \in X}$ is a weight matrix satisfying \eqref{Bcond}, then 
  \begin{align} \label{Brep1}
    \La^{(\fM)}_n = \bigcup \big\{\La^{\{L\}}_n : L \{\lhd) \fM,~ \ell_k^{1/k} \to \infty\big\}. 
  \end{align}
\end{proposition}

\begin{proof}
  Let us show the nontrivial inclusion $\subseteq$. Let $a = (a_\al) \in \La^{(\fM)}_n$ and set 
  $L_k := \max\{\max_{|\al|=k} |a_\al|, k!\}$. Then $L \{\lhd) \fM$ as $a \in \La^{(\fM)}_n$ and by \eqref{Bcond}. 
  Lemma~\ref{LNM} provides a positive sequence $N=(N_k)$ such that $L \lhd N \{\lhd) \fM$, and thus $a \in \La^{(N)}_n \subseteq \La^{\{N\}}_n$. 
  In particular, $(k!)_k \lhd N$, that is $n_k^{1/k} \to \infty$. The proof is complete. 
\end{proof}

\begin{remark}
  (1) The proof of Proposition \ref{Beurlingrepresentation} actually shows that
  \begin{align} \label{Brep2}
    \La^{(\fM)}_n = \bigcup \big\{\La^{(L)}_n : L \{\lhd) \fM,~ \ell_k^{1/k} \to \infty\big\}. 
  \end{align}
  
  (2) It is easy to see that \eqref{Brep1} and \eqref{Brep2} hold with $\La^{(\fM)}_n$ and $\La^{[L]}_n$ replaced by 
  $\cE^{(\fM)}(K)$ and $\cE^{[L]}(K)$ for any compact $K \subseteq \R^n$, where
  \begin{align*}
    \cE^{\{\fM\}}(K) &:= \on{ind}_{x > 0} \on{ind}_{\rh>0} \cE^{M^x}_\rh(K), \\
    \cE^{(\fM)}(K) &:= \on{proj}_{x > 0} \on{proj}_{\rh>0} \cE^{M^x}_\rh(K).
  \end{align*}
  
  (3) In the situation of Lemma \ref{LNM} it is sometimes possible to transfer properties of $\fM$ to $N$,   
  just as in Corollary \ref{corLNM}. Another instance is the following: if $L \{\lhd) \fM$, where $L$ satisfies \eqref{M3} and  
  $\cE^{(\fM)}$ is non-quasianalytic, then there is a log-convex non-quasianalytic $N$ satisfying $L \lhd N \{\lhd) \fM$.

  (4) In analogy to the inductive representations in \eqref{Brep1} and \eqref{Brep2}, there are 
  projective representations of the form 
  \begin{align*}
       \La^{\{\fM\}}_n = \bigcap \La^{(L)}_n  = \bigcap \La^{\{L\}}_n, 
  \end{align*} 
  where the intersections are taken over all weight sequences $L$ with $\fM \{\lhd) L$; similarly, for $\cE^{\{\fM\}}(U)$ and 
  $\cE^{[L]}(U)$ and any open $U \subseteq \R^n$. See \cite[Proposition 9.4.4]{Schindl14}.
\end{remark}

Now we are ready to show our main result in the Beurling case.

\begin{theorem} \label{matrixB}
  Let $\fM=\{M^x\}_{x \in X}$ be a quasianalytic weight matrix such that $\cO_{0,n} \subsetneq \cE^{(\fM)}_{0,n}$. 
  Then there exist elements in $\La^{(\fM)}_n$ that are not contained in $j^\infty \cE^{\{\fN\}}_{0,n}$ for any quasianalytic 
  weight matrix $\fN = \{N^y\}_{y \in Y}$. 
\end{theorem}

In particular, 
there are elements in $\La^{(\fM)}_n$ not contained in $j^\infty \cE^{(\fN)}_{0,n}$ for any quasianalytic 
weight matrix $\fN = \{N^y\}_{y \in Y}$,  
since always $\cE^{(\fN)} \subseteq \cE^{\{\fN\}}$.

\begin{proof}
  Let $L=(L_k)$ be a positive sequence satisfying $L \{\lhd) \fM$ and $\ell_k^{1/k} \to \infty$ which exists by Proposition 
  \ref{Beurlingrepresentation}. Let $\underline L=(\underline L_k)$ denote the log-convex minorant of $L$. Then still
  $\underline L \{\lhd) \fM$ and $\underline \ell_k^{1/k} \to \infty$, by \cite[Theorem 2.15]{RainerSchindl12}. It follows 
  that $\underline L$ is a quasianalytic weight sequence, since $\cE^{(\fM)}$ is quasianalytic. 
  Thanks to $\underline \ell_k^{1/k} \to \infty$ we have $\cO_{0,n} \subsetneq \cE^{\{\underline L\}}_{0,n}$. 

  Now the assertion is a direct consequence of 
  Theorem \ref{seqRoumieu} (or Theorem \ref{matrixR}) applied to $\underline L$.
\end{proof}

By Theorem \ref{representation} and Proposition \ref{repLa}, we immediately get the following corollary.
 
\begin{corollary}
  Let $\om$ be a quasianalytic weight function such that $\cO_{0,n} \subsetneq \cE^{(\om)}_{0,n}$. 
  Then there exist elements in $\La^{(\om)}_n$ that are not contained in $j^\infty \cE^{\{\si\}}_{0,n}$ for any quasianalytic 
  weight function $\si$. 
\end{corollary}

The strict inclusion $\cO_{0,n} \subsetneq \cE^{(\om)}_{0,n}$ holds if and only if $\om(t) =o(t)$ as $t \to \infty$; cf. 
\cite[Corollary 5.17(3)]{RainerSchindl12}.


\def\cprime{$'$}
\providecommand{\bysame}{\leavevmode\hbox to3em{\hrulefill}\thinspace}
\providecommand{\MR}{\relax\ifhmode\unskip\space\fi MR }
\providecommand{\MRhref}[2]{%
  \href{http://www.ams.org/mathscinet-getitem?mr=#1}{#2}
}
\providecommand{\href}[2]{#2}

\end{document}